\documentclass{article}
\usepackage{amssymb,amsmath,amsthm}

\newtheorem{theorem}{Theorem}
\newtheorem{lemma}[theorem]{Lemma}
\newtheorem{claim}[theorem]{Claim}

\newtheorem{corollary}[theorem]{Corollary}

\theoremstyle{definition}

\let\eps\varepsilon
\let\theta\vartheta
\let\phi\varphi
\newcommand\M{\mathcal M}
\newcommand\LL{\mathcal L}
\newcommand\cl{\mathop{\mathsf{cl}}}
\newcommand\ING{\mathop{\mathsf{ING}}\nolimits}
\newcommand\COMM{\mathop{\mathsf{COMM}}\nolimits}
\newcommand\incmp{\mathbin{\|}}
\newcommand\contr{\mathop{\backslash}}

\title{\bf Sticky matroids and convolution}
\author{Laszlo Csirmaz%
\thanks{Central European University}
\thanks{e-mail:~csirmaz@renyi.hu}}
\date{\small\it Dedicated to the memory of Frantisek Mat\'u\v s}

\begin{document}
\maketitle
\begin{abstract}
Motivated by the characterization of the lattice of cyclic flats of a matroid,
the convolution of a ranked lattice and a discrete measure is defined,
generalizing polymatroid convolution. Using the convolution technique we prove
that if a matroid, or polymatroid, has a non-principal modular cut then it is not sticky.
A similar statement for matroids has been proved in \cite{hoch-wilhelmi} using
different technique.



\end{abstract}

\section{Introduction}\label{sec:introduction}

The main purpose of this paper is to provide a proof of the following
statement: \emph{if a (poly)matroid $M$ has a non-principal modular cut then
it is not sticky}. A similar statement for matroids was claimed as Theorem 9
of \cite{hoch-wilhelmi}. Our main tool is a novel convolution-like
construction of polymatroids using a ranked lattice and a discrete measure.
The construction is a common generalization of polymatroid convolution
\cite{Lovasz} and the characterization of cyclic flats \cite{cyclic-flats,
sims}.

We assume familiarity with matroid theory, the main reference is
\cite{oxley}. Most of our results are stated in terms of polymatroids
introduced by Edmonds \cite{edmonds}. Several notions from matroid theory
generalize to polymatroids, and results can frequently be transferred back
to matroids, see \cite{Lovasz}.

All sets in this paper are finite. Sets are denoted by capital letters, such
as $A$, $B$, $K$, etc., their elements by lower case letters. As customary,
curly brackets around singletons are frequently omitted as well as the union
sign. Thus $abAB$ denotes the set $\{a,b\}\cup A\cup B$.

\section{Preliminaries}

This section recalls some basic facts about polymatroids and related
operations. Some are straightforward generalizations from matroid theory,
others need careful tailoring.

\subsection{Polymatroids}\label{subsec:polymatroid}

A \emph{polymatroid} $(f,M)$ is a non-negative, monotone and submodular
real-valued function $f$ defined on the non-empty subsets of the finite set
$M$. Here $M$ is the \emph{ground set}, and $f$ is the \emph{rank function}.
The polymatroid is \emph{integer} if all ranks are integer. An integer
polymatroid is a \emph{matroid} if the rank of singletons are either zero or
one. For details please consult \cite{edmonds,Lovasz,M.entreg,oxley}.

Let $(f,M)$ be a polymatroid, The subset $F\subseteq M$ is a \emph{flat} if
proper extensions of $F$ have strictly larger ranks. The intersection of
flats is a flat, and the \emph{closure} of $A\subseteq M$, denoted by
$\cl_f(A)$, or simply by $\cl(A)$ if the polymatroid is clear from the
context, is the smallest flat containing $A$ (the intersection of
all flats containing $A$). The ground set $M$ is always
a flat, and flats of a polymatroid form a \emph{lattice}, where the meet of
two lattice elements is their intersection, and the join is the closure of
their union.

The flat $C\subseteq M$ is \emph{cyclic} if for all $i\in C$ either $f(i)=0$
or $f(C)-f(C{-}i)<f(i)$. When $(f,M)$ is a matroid then this definition is
equivalent to the original one, namely that $C$ is a union of cycles in the
matroid.

The \emph{modular defect} $\delta(A,B)$ of a pair of subsets $A,B\subseteq
M$ is defined as
$$
    \delta(A,B) = f(A)+f(B)-f(A\cap B)-f(A\cup B).
$$
By submodularity, this is always non-negative. If it equals
zero, the pair $(A,B)$ is a \emph{modular pair}.

\smallskip

A (discrete) \emph{measure} on the finite set $M$ is a
non-negative, additive function defined on subsets of $M$. Such a measure $\mu$ is
determined by its value on singletons:
$$
    \mu(A)={\textstyle\sum}\,\{\mu(a)\,:\,a\in A\}.
$$
For a polymatroid $(f,M)$ the associated measure $\mu_f$ is defined by
$\mu_f(a)=f(a)$ for singletons $a\in M$. The inequality $f(A)\le\mu_f(A)$ is
a consequence of submodularity. If this inequality holds with 
equality for all $A\subseteq M$, then the polymatroid $(f,M)$ is \emph{modular} 
\cite{Lovasz,M.entreg}. Modularity for matroids, however, is defined 
differently in \cite{oxley}, which is called flat-modular here.
The polymatroid is \emph{flat-modular} if every
pair of its flats is a modular pair. Modular polymatroids are flat-modular,
but the converse is not true, the simplest example is the matroid on
two elements with rank function
$$
   f(a)=f(b)=f(ab)=1.
$$
A collection $\M$ of flats is a \emph{modular cut} if
\begin{itemize}\setlength\itemsep{2pt}
\item[(i)] $\M$ is closed upward: if $F_1\in\M$ and $F_1\subseteq F_2$, then $F_2\in\M$,
\item[(ii)] if the pair of elements $F_1,F_2$ of $\M$ is modular, then
$F_1\cap F_2\in\M$.
\end{itemize}
The modular cut $\M$ is \emph{principal} if the intersection of all elements
of $\M$ is an element of $\M$ (including the case when $\M$ has no element
at all). Thus if $\M$ is not principal, then there are $F_1,F_2\in\M$ such
that $F_1\cap F_2\notin\M$ (and then $\delta(F_1,F_2)$ must be positive).
For such a non-principal modular cut, $\delta(\M)$ is the minimal
modular defect of those pairs $F_1,F_2\in\M$ for which $F_1\cap F_2\notin\M$.

Given the flats $F_1, F_2$, the smallest modular cut containing both of them
is \emph{generated by $F_1$ and $F_2$}, and is denoted by $\M(F_1,F_2)$. 
This is a sound definition as the 
intersection of any collection of modular cuts is a modular cut. Let
$\M$ be a non-principal modular cut, and choose the pair $F_1,F_2\in\M$,
$F_1\cap F_2\notin\M$ 
such that $\delta(\M)=\delta(F_1.F_2)$. As
$\M(F_1,F_2)$ is contained in $\M$ and $F_1\cap F_2\notin\M$, the same
intersection is not in $\M(F_1,F_2)$ either. Thus $\M(F_1,F_2)$ is not
principal, and $\delta(\M)=\delta(\M(F_1,F_2))$. The following lemma assumes
this situation.

\begin{lemma}\label{lemma:mod-cut-estimated}
Suppose $\M=\M(F_1,F_2)$ as above, and let $S=F_1\cap F_2$.
Then for all $F\in\M$, {\upshape(a)} $S\subseteq
F$; {\upshape(b)} $f(F)-f(S)>\delta$.
\end{lemma}

\begin{proof}
As $S$ is a flat, and both $F_1$ and $F_2$ are elements of the principal cut
generated by $S$, every element of $\M(F_1,F_2)$ is there. This
proves (a). For (b) choose $F\in\M$ with $f(F)-f(S)$ minimal. 
Then either $F_1$ or $F_2$ is incomparable to $F$. Indeed, if $F$ is equal
to one of them, then the other one works. Otherwise neither $F_1$ nor $F_2$
is below $F$, and $F$ cannot be below both $F_1$ and $F_2$. So let $G$ be
either $F_1$ or $F_2$ such that $F$ and $G$ are incomparable. 
The intersection $F\cap G$ is not in $\M$ (since $f(F\cap
G)<f(F)$ and then $f(F)-f(S)$ cannot be minimal), $S\subseteq F\cap G$ by
a) since $F,G\in\M$, and
\begin{align*}
  \delta&\le \delta(F,G) = {}\\
        &=\big(f(F)-f(F\cap G)\big) - \big(f(F\cup G)-f(G)\big) <{}\\
        &<f(F)-f(F\cap G)\le f(F)-f(S),
\end{align*}
since the flat $\cl(F\cup G)$ is different from $G$. This proves the claim.
\end{proof}

The following example illustrates a non-principal modular cut. The
polymatroid is defined on three elements $\{a,b,c\}$. The rank function is
$f(a)=f(b)=f(c)=2$, $f(ab)=3$, and the rank of all other subsets are $4$. The 
modular cut $\M=\{a,b,ab,abc\}$ is generated by $a$ and $b$, and
$\delta(\M)=\delta(a,b)=2+2-3=1$. This is a counterexample to the
polymatroid version of \cite[Lemma 2]{hoch-wilhelmi} as $ab$ is the 
only hyperplane in $\M$ and it
is not part of any pair with positive modular defect.

\subsection{Extensions, factors, contracts}\label{subsec:extension}

The polymatroid $(g,N)$ is an \emph{extension} of $(f,M)$ if $N\supseteq M$,
and $f(A)=g(A)$ for all $A\subseteq M$. This is a \emph{one-point extension}
if $N{-}M$ has a single element. Similarly to the matroid case, there is a 
strong connection between modular cuts and one-point extensions, but this is
not a one-to-one connection.

\smallskip

The \emph{sticky matroid conjecture}, due to Poljak and Turzik
\cite{poljak-turzik}, concerns the question whether any two extension of a
matroid can be glued together along their common part -- such an extension
is called their \emph{amalgam}. Matroids with this property are
called \emph{sticky}. Flat-modular matroids are sticky, see \cite{oxley}, the
proof generalizes to polymatroids as well. The conjecture is that no other
(poly)matroids are sticky. Interestingly, the conjecture for matroids and
polymatroids are not equivalent: matroids $M_1$ and $M_2$ might have a
joint polymatroid extension but not a matroid extension; alas, no such a
pair of matroids is known. Nevertheless
techniques attacking the problem give similar results for both cases.

\smallskip

Let $\cong$ be an equivalence relation on $M$. Let $N=M/{\cong}$ be the
collection of equivalence classes, and $\phi:M\to N$ be the map which
assigns to each element $i\in M$ its equivalence class.
The \emph{factor of $(f,M)$ by $\cong$}, denoted as $(f,M)/{\cong}$,
is the pair $(g,N)$ where $g$ assigns the value $g: A\mapsto f(\phi^{-1}(A))$
to each $A\subseteq N$. The factor of a factor is
also a factor, thus typically it is enough to consider equivalence relations
with only one non-singleton class. Any factor of a polymatroid
is also a polymatroid. By a result of Helgason \cite{helgason} every integer
polymatroid is a factor of a matroid. This claim follows from Theorem
\ref{thm:factor-extension} of Section \ref{sec:applications}, which also 
implies that factors of sticky polymatroids are sticky.

\smallskip

Let $X\subset M$. The \emph{contract of ($f,M)$ along $X$}, denoted as
$(f,M)\contr X$, is the pair $(g,M{-}X)$ defined for $A\subseteq M{-}X$ as
$$
    g: A \mapsto f(AX)-f(X).
$$
The contract of a polymatroid is a polymatroid. By Corollary
\ref{corr:contracts-are-sticky} contracts of sticky polymatroids are
sticky. Instead taking the contract directly, one can first compute the
factor with the subset $X$ as the only non-singleton class, and then take
the contract where $X$ is a singleton. This approach will be followed in
Section \ref{sec:applications}.

\subsection{Our contribution}\label{subsec:contribution}

Theorem 9 in \cite{hoch-wilhelmi} states that if the matroid $M$ has a
non-principal modular cut, then it is not sticky. 
Our main result is a generalization to polymatroids.
\begin{theorem}\label{thm:main-result}
A polymatroid with a non-principal modular cut is not sticky.
\end{theorem}
\noindent
Given such a polymatroid we construct two extensions which have no amalgam.
If we start from a matroid, then the extensions are integer valued -- thus
can be extended further to be matroids which still have no amalgam, yielding
the quated result from \cite{hoch-wilhelmi}, see Remark 1. The main
technique is convolution of a ranked lattice and a (discrete) measure
defined in Section \ref{sec:convolution}.

\medskip

The rest of the paper is organized as follows. Section
\ref{sec:cyclic-flats} proves properties of cyclic flats which will be used
later. The convolution of a ranked lattice and a measure is defined in
Section \ref{sec:convolution} along with the proof of some properties of
the convolution. The method is used to prove the existence of several
interesting extensions in Section \ref{sec:applications}. The proof
of Theorem \ref{thm:main-result} is given in Section \ref{sec:main-theorem},
and Section \ref{sec:conclusion} concludes the paper.

\section{Cyclic flats}\label{sec:cyclic-flats}

A flat $C$ of the polymatroid $(f,M)$ is \emph{cyclic} if for all $i\in C$ either
$f(i)=0$ or $f(C)-f(C{-}i)<f(i)$. As remarked in Section
\ref{subsec:polymatroid}, for
matroids this definition is equivalent to the condition that $C$ is a union of cycles.
Cyclic flats turned out to be very useful in proving that certain matroids
are not sticky \cite{bonin}. As cyclic flats are not part of the standard
repertoire of polymatroids, properties used later are proved here.

\begin{lemma}\label{lemma:basic-cyclic-flats}
Every flat $F\subseteq M$ contains a unique maximal cyclic flat $C\subseteq F$.
Moreover, for every $C\subseteq A\subseteq F$ we have
$$
    f(A) = f(C)+\mu_f(A{-}C).
$$
\end{lemma}
\begin{proof}
First we show that $F$ contains a maximal cyclic flat with the given property, then
we show that the maximal cyclic flat inside $F$ is unique.

Start with $F_1=F$. For $F_j$ with $j\ge 1$, if there is an element of
$x_j\in F_j$ such that $f(x_j)>0$ and $f(F_j)=f(x_j)+f(F_j{-}x_j)$, then let
$F_{j+1}=F_j{-}x_j$, otherwise stop. Submodularity implies that for each
$k<j$ $f(F_jx_k) = f(F_j)+f(x_k)$, thus all $F_j$ is a flat (as it is
contained inside $F$). It is also clear that the last $F_j$ is cyclic. As no
proper extension of $F_j$ inside $F$ is cyclic, it is a maximal cyclic flat
with the desired property.

Second, suppose $C\subseteq F$ is a maximal cyclic flat. If $i\in C$ with
$f(i)>0$ then $f(F)-f(F{-}i)\le f(C)-f(C{-}i) < f(i)$. Consequently, $C$ is
a subset of $F_1=\{i\in F: f(i)=0$ or $f(F)-f(F{-}i)<f(i)\}$. Define
similarly the sets $F_1\supseteq F_2\supseteq\cdots$. It is clear that all
$F_j$ is a flat, and when $F_j=F_{j+1}$ then it is also cyclic. As it
contains $C$, it must be equal to $C$.
\end{proof}

\begin{claim}\label{claim:cyclic-lattice}
The cyclic flats of the polymatroid $(f,M)$ form a sublattice of the lattice
of all subsets of $M$.
\end{claim}
\begin{proof}
Let $C_1$ and $C_2$ be cyclic flats. Then $C_1\cap C_2$ is a flat, which
contains a unique largest cyclic flat by Lemma
\ref{lemma:basic-cyclic-flats} above. This is the largest lower bound of $C_1$ and
$C_2$. The smallest upper bound is $C=\cl_f(C_1\cup C_2)$. First, this is
clearly a flat. Second, this is cyclic: if $i\in C{-}(C_1\cup C_2)$, 
then $f(C)-f(C{-}i)=0$. If, say, 
$i\in C_1$ and $f(i)>0$, then submodularity gives $f(C)-f(C{-}i)\le
f(C_1)-f(C_1{-}i)<f(i)$ proving that this closure is indeed the smallest cyclic
flat containing both $C_1$ and $C_2$.
\end{proof}

\section{Ranked lattices and convolution}\label{sec:convolution}

Let $\LL$ be a sublattice of the lattice of the subsets of $M$. The join and
meet of $Z_1,Z_2\in\LL$ is denoted by $Z_1\lor Z_2\supseteq Z_1\cup Z_2$ and 
$Z_1\land Z_2\subseteq Z_1\cap Z_2$, respectively (the ordering in $\LL$ 
is inherited from the subset relation).
We write $Z_1<Z_2$ if $Z_1$ is strictly below $Z_2$, and $Z_1\incmp Z_2$ if
$Z_2$ and $Z_2$ are
incomparable (none of them is below the other). The lattice $\LL$ is
\emph{ranked} if each element $Z\in\LL$ has a non-negative rank $r(Z)$. 
Elements of $\LL$ are typically denoted by 
$Z$ (with or without indices), and the lattice rank function by
$r$.

Given the ranked lattice $(r,\LL)$ and the measure $\mu$, their
\emph{convolution}, denoted by $r*\mu$,
assigns a non-negative value to subsets of $M$ as follows:
\begin{equation}\label{eq:convolution}
  r*\mu \,:\,    A \mapsto \min\,\{\: r(Z)+\mu(A{-}Z)\,:\, Z\in\LL \}.
\end{equation}
Typically we write $r'$ in place of $r*\mu$.
When $\LL$ contains all subsets of $M$ and
$(r,M)$ is a polymatroid, then (\ref{eq:convolution}) is equivalent to
the usual definition of the convolution of a polymatroid and a modular
polymatroid, see \cite{Lovasz,M.entreg}.
\begin{theorem}\label{thm:convolution}
Let $(r,\LL)$ be a ranked lattice and $\mu$ be a measure on $M$. Suppose
\begin{equation}\label{eq:convolution.1}
   r(Z_1)+r(Z_2)\ge r(Z_1\land Z_2)+r(Z_1\lor Z_2) + \mu(Z_1\cap
Z_2{-}Z_1\land Z_2)
\end{equation}
for every pair of incomparable elements $Z_1,Z_2\in\LL$.
Then the convolution $r'=r*\mu$ defines a polymatroid on $M$.
\end{theorem}
\begin{proof}
First observe that for arbitrary subsets $A$, $B$, $Z_A$, $Z_B$ we have
\begin{equation}\label{eq:sub-basic}
\mu(A{-}Z_A)+\mu(B{-}Z_B) \ge
 \mu(A\cap B{-}Z_A\cap Z_B) + \mu(A\cup B{-}Z_A\cup Z_B).
\end{equation}
Indeed, if $i\in A\cap B{-}Z_A\cap Z_B$, then $i$ is in both $A$ and $B$,
and not in either $Z_A$ or $Z_B$, thus $i$ is an element of either $A{-}Z_A$
or $B{-}Z_B$. Similarly, if $i\in A\cup B{-}Z_A\cup Z_B$, then $i$ is not in
$Z_A$ and not in $Z_B$, but it is in either $A$ or $B$, thus it also appears
on the left hand side. Finally, if $i$ is a member of both sets on the right
hand side, then $i$ is
in $A\cap B$, and $i$ is not is $Z_A\cup Z_B$, thus $i$ is in both sets on
the left hand side.

Clearly, the convolution $r'$ takes non-negative values only. Now suppose
$r'(A)=r(Z_A)+\mu(A{-}Z_A)$ and $r'(B)=r(Z_B)+\mu(B{-}Z_B)$. Monotonicity of
$r'$ follows from the fact that for $A\subseteq B$,
$$
   r'(A) \le r(Z_B)+\mu(A{-}Z_B) \le r(Z_B)+\mu(B{-}Z_B) = r'(B).
$$
To check submodularity, consider first the case when $Z_A<Z_B$. In this case
\begin{align*}
   r'(A\cap B) &\le r(Z_A)+\mu(A\cap B{-}Z_A) = r(Z_A)+\mu(A\cap B{-}Z_A\cap
Z_B), \\
   r'(A\cup B) &\le r(Z_B)+\mu(A\cup B{-}Z_B) = r(Z_B)+\mu(A\cup B{-}Z_A\cup
Z_B),
\end{align*}
thus $r'(A)+r'(B)\ge r'(A\cap B)+r'(A\cup B)$ follows from
(\ref{eq:sub-basic}). When $Z_A$ and $Z_B$ are incomparable, we use
$Z_A\land Z_B$ and $Z_A\lor Z_B$ to estimate $r'(A\cap B)$ and $r'(A\cup
B)$, respectively as follows:
\begin{align*}
  r'(A\cap B) &\le r(Z_A\land Z_B) + \mu(A\cap B{-}Z_A\land Z_B), \\
  r'(A\cup B) &\le r(Z_A\lor Z_B) + \mu(A\cup B{-}Z_A\lor Z_B).
\end{align*}
Using condition (\ref{eq:convolution.1}) the submodularity follows if we
have
\begin{align*}
 &  \mu(A{-}Z_A)+\mu(B{-}Z_B) + \mu(Z_A\cap Z_B{-}Z_A\land Z_B) \\
 & ~~~~~~~~~ \ge  \mu(A\cap B{-}Z_A\land Z_B) +  \mu(A\cup B{-}Z_A\lor Z_B) .
\end{align*}
As $\mu(A\cap B{-}Z_A\land Z_B)\le \mu( A\cap B{-}Z_A\cap Z_B\big)
+ \mu(Z_A\cap Z_B{-}Z_A\land Z_B)$ (this is disjoint union), and $\mu(A\cup
B{-}Z_A\lor Z_B) \le \mu(A\cup B{-}Z_A\cup Z_B)$ (as $Z_A\cup Z_B$ is a
subset of $Z_A\lor Z_B$), inequality (\ref{eq:sub-basic}) gives the claim.
\end{proof}

The following lemma gives conditions for the convolution to extend the
lattice rank function. Then we look at the case when the lattice rank is
defined partially from a polymatroid, and under what conditions will the
convolution keep the polymatroid rank.

\begin{lemma}\label{lemma:keep-lattice-rank}
Let $(r,\LL)$ be a ranked lattice and $\mu$ be a measure on $M$. Assume that
for $Z_1,Z_2\in\LL$, if $Z_1<Z_2$ then
\begin{equation}\label{eq:convolution.2}
   0\le r(Z_2)-r(Z_1) \le \mu(Z_2{-}Z_1);
\end{equation}
and if $Z_1$ and $Z_2$ are incomparable, then
{\upshape(\ref{eq:convolution.1})} holds. Then $(r',M)$ is a polymatroid,
and $r'(A)=r(A)$ for every $A\in\LL$.
\end{lemma}
\begin{proof}
By Theorem \ref{thm:convolution} $(r',M)$ is a polymatroid, thus we focus on
the second claim.
As $A\in\LL$ it is clear that $r'(A)\le r(A)$, thus we need to show $r(A)\le
r(Z)+\mu(Z{-}A)$ for every $Z\in\LL$. If $A<Z$ then (\ref{eq:convolution.2})
gives $r(A)\le r(Z)$. If $Z<A$ then $r(A)\le r(Z)+\mu(Z{-}A)$ again by
(\ref{eq:convolution.2}). Thus we can assume $A$ and $Z$ are incomparable. As
$r(A)\le r(A\lor Z)$ by (\ref{eq:convolution.2}), it is enough to show that
$r(A\lor Z)\le r(Z)+\mu(A{-}Z)$. Applying (\ref{eq:convolution.2}) to $A$
and $A\land Z$, and (\ref{eq:convolution.1}) to $A$ and $Z$ we get
\begin{align*}
   r(A) & \le r(A\land Z)+\mu(A{-}A\land Z), \\
 r(A\land Z)+r(A\lor Z) &\le r(A)+r(Z) - \mu(A\cap Z{-}A\land Z).
\end{align*}
Adding them up we get the required inequality.
\end{proof}

\begin{lemma}\label{lemma:conv-extend-poly}
Let $(f,M)$ be a polymatroid, $N\supseteq M$, and $(r,\LL)$ be a ranked lattice
on subsets of $N$. Suppose the following conditions hold:
\begin{align}
  &\mbox{for all $Z\in\LL$, $Z\cap
M\in\LL$,\hspace{0.5\textwidth}}\label{cond:1}\\
  &\mbox{all cyclic flats of $M$ are in $\LL$},\label{cond:2}\\
  &\mbox{if $Z_1<Z_2$ then $r(Z_1)\le r(Z_2)$},\label{cond:3} \\
  &\mbox{if $Z\in\LL$, $Z\subseteq M$ then $r(Z)=f(Z)$},\label{cond:4}\\
  &\mbox{for all $a\in M$ we have $\mu(a)=f(a)$}.\label{cond:5}
\end{align}
Then $r'(A)=f(A)$ for all $A\subseteq M$.
\end{lemma}
\begin{proof}
Condition (\ref{cond:1}) says that $\LL$ restricted to subsets of $M$ is a
ranked lattice. As $Z\cap M\in\LL$ is below $Z$, for $A\subseteq M$ we have
$$
    r(Z)+\mu(A{-}Z) \ge r(Z\cap M)+\mu(A{-}Z\cap M)
$$
by (\ref{cond:3}), thus the lattice element which minimizes
$r(Z)+\mu(A{-}Z)$ is a subset of $M$. For every $Z\subseteq M$,
$$
    f(A) \le f(Z)+f(A{-}Z) \le f(Z)+\mu_f(A{-}Z)=r(Z)+\mu(A{-}Z)
$$
by (\ref{cond:4}) and (\ref{cond:5}), which implies $r'(A)\ge f(A)$. To 
show that they are
equal, we need to exhibit a $Z\in\LL$ with equality here. This can be done by
condition (\ref{cond:2}) and Lemma \ref{lemma:basic-cyclic-flats}: 
choose $Z\in\LL$ to be the maximal cyclic flat inside $\cl(A)$.
\end{proof}

\section{Applications}\label{sec:applications}

In this section we show how lattice convolution can be used to create
polymatroid extensions with desired properties. Theorem
\ref{thm:factor-extension} essentially says that an extension of a
polymatroid factor
is a factor of an extension. Let $\cong$ be an equivalence
relation on $M$; this relation extends to any $N\supseteq M$ by stipulating that
elements of $N{-}M$ are equivalent to themselves only.

\begin{theorem}\label{thm:factor-extension}
Let $(f,M)$ be a polymatroid and $(f',M')$ be the factor $(f,M)/{\cong}$. Let
$(g',N')$ be an extension of $(f',M')$. Then there is an extension $(g,N)$
of $(f,M)$ such that $(g',N')$ is the factor $(g,N)/{\cong}$.
\end{theorem} 
\begin{proof}
Let $\phi:M\to M'$ map elements of $M$ to their equivalence classes. Let
$N=M\cup(N'{-}M')$, and extend $\phi$ to $N$ by stipulating that it is the 
identity on $N'{-}M'$. The
lattice $\LL$ consists of subsets $Z\subseteq N$ for which
\begin{itemize}
\item[a)] either $Z\subseteq M$;
\item[b)] or $Z=\phi^{-1}(\phi(Z))$, i.e., $Z$ contains complete equivalence
classes only.
\end{itemize}
$\LL$ is a lattice, and $Z_1\land Z_2=Z_1\cap Z_2$. If
both $Z_1$ and $Z_2$ satisfy a) or none of them satisfies a), then $Z_1\lor
Z_2=Z_1\cup Z_2$, otherwise $Z_1\lor Z_2=\phi^{-1}(\phi(Z_1\cup Z_2))$.
Define the rank $r(Z)$ as
$$
   r(Z) = \begin{cases} f(Z) & \mbox{if $Z\subseteq M$}, \\
                        g'(\phi(Z)) & \mbox{if $Z=\phi^{-1}(\phi(Z))$}.
          \end{cases}
$$
If $Z$ satisfies both conditions, then the two lines give the same value
$f(Z)=f'(\phi(Z))=g'(\phi(Z))$ as $g'$ is an extension of $f'$. Define
$\mu(a)=f(a)$ if $a\in M$, and $\mu(a)=g'(a)$ if $a\in N{-}M=N'{-}M'$.
Let $(g,N)$ be the convolution of $(r,\LL)$ and $\mu$. We claim that both
conditions of Lemma \ref{lemma:keep-lattice-rank} hold. If both $Z_1$ and
$Z_2\in\LL$ satisfy a) then they follow from the fact that $(f,M)$ is a
polymatroid and $\mu(a)=f(a)$. If both $Z_1$ and $Z_2$ satisfies b), then 
one uses the fact
that $(g',N')$ is a polymatroid. Finally, let $Z_1\subseteq M$, and
$Z_2=\phi^{-1}(\phi(Z_2))$, and set $Z_1'=\phi^{-1}(\phi(Z_1))\subseteq M$.
Then $Z_1\lor Z_2=Z_1'\cup Z_2$, and the submodularity of $g'$ gives
$$
  r(Z_1'\cup Z_2) - r(Z_2) \le r(Z_1') - r(Z_1'\cap Z_2),
$$
while submodularity of $f$ yields
$$
   r(Z_1')-r(Z_1'\cap Z_2) \le r(Z_1)-r(Z_1\cap Z_2).
$$
Together they give (\ref{eq:convolution.1}) for this case as well. Similar 
calculation shows that (\ref{eq:convolution.2}) also holds. Thus for all 
$Z\in\LL$ we have $g(Z)=r(Z)$, which means that a) $g$ is an extension
of $f$, and b) $g/{\cong}$ is the same as $g'$, as was claimed.
\end{proof}

\begin{corollary}\label{corr:factors-are-sticky}
Factors of a sticky polymatroid are sticky.
\end{corollary}
\begin{proof}
Let $f'$ be a factor of $f$, and $g_1', g_2'$ be two extensions of $f'$.
Then there are extensions $g_1, g_2$ of $f$ such that $g_i'$ is a factor of
$g_i$ using the same equivalence relation. If $f$ is sticky, then $g_1,g_2$ has an
amalgam $g$, and then the factor of $g$ is an amalgam of $g'_1$ and $g'_2$.
\end{proof}

Helgason's theorem \cite{helgason} is another consequence of Theorem
\ref{thm:factor-extension}. We state a more general statement, Helgason's
original construction is the special case when the matroids $(r_i,P_i)$ are
the free matroids on $f(i)$ elements.
\begin{corollary}\label{corr:poly-to-matroid}
Let $(f,M)$ be an integer polymatroid, and for each $i\in M$ let $(r_i,P_i)$ be 
rank $f(i)$ matroid with disjoint ground sets $P_i$. Let
$N=\bigcup\,\{P_i:i\in M\}$. There is matroid $(g,N)$ such that $(f,M)$ is
a factor of $(g,N)$; moreover $(g,N)$ restricted to $P_i$ is isomorphic to
$(r_i,P_i)$.
\end{corollary}
\begin{proof}
Replace each $i\in M$ by $P_i$ one after the other. The one-point polymatroid $(f_i,\{i\})$
with the rank function $f_i(i)=r_i(P_i)=f(i)$ is a factor of $(r_i,P_i)$, and
$(f,M)$ is clearly an extension of this one-point polymatroid. By Theorem
\ref{thm:factor-extension} there is an extension of $(r_i,P_i)$ such that
$(f,M)$ is isomorphic to the factor when $P_i$ is replaced by a single
point.
\end{proof}

The next construction gives similar results for contracts of polymatroids.
\begin{theorem}\label{thm:contract-extension}
Let $(f,M)$ be a polymatroid, $X\subset M$, and $(f',M')$ be the contract
$(f,M)\contr X$. Let $(g',N')$ be an extension of $(f',M')$. There is
an extension $(g,N)$ of $(f,M)$ such that $(g',N')=(g,N)\contr X$.
\end{theorem}
\begin{proof}
As remarked at the end of Section \ref{subsec:extension}, it suffices to
consider the case when $X$ is a singleton, say $X=\{x\}$. Then
$M=M'\cup\{x\}$,
$N=N'\cup\{x\}$, and $f'(A)=f(Ax)-f(x)$ for all $A\subseteq M'$. The lattice $\LL$
consists of subsets $Z\subseteq N$ for which
\begin{itemize}
\item[a)] either $Z\subseteq M$, or
\item[b)] $x\in Z$.
\end{itemize}
Clearly, the join and meet is the union and intersection, respectively.
Define the rank $r(Z)$ as follows:
$$
   r(Z)= \begin{cases} f(Z) & \mbox{if $Z\subseteq M$}, \\
                       g'(Z{-}x)+f(x) & \mbox{if $x\in Z$}.
         \end{cases}
$$
As before, if $Z\in\LL$ satisfies both a) and b), then
$g'(Z{-}x)=f'(Z{-}x)=f(Z)-f(x)$, i.e., both lines give the same value.
Define $\mu(i)=f(i)$ if $i\in M$, and $\mu(i)=g'(i)$ when $i\in N'{-}M'$.
Let $(g,N)$ be the convolution of $(r,\LL)$ and $\mu$. Similarly to the
proof of Theorem \ref{thm:factor-extension}, if both $Z_1, Z_2\in\LL$
satisfy a), or both satisfy b) then the conditions of Lemma
\ref{lemma:keep-lattice-rank} hold. Suppose $Z_1\subseteq M$ and $x\in Z_2$.
Cases $Z_1<Z_2$ or $Z_2<Z_1$ are handled before. If $Z_1$ and $Z_2$ are
incomparable, then
$$
   r(Z_1\cup Z_2)-r(Z_2) \ge r(Z_1x)-r((Z_1\cap Z_2)x) \ge r(Z_1)-r(Z_1\cap
Z_2)
$$
applying previously settled cases. Thus for all $Z\in\LL$ we have
$r(Z)=g(Z)$ which proves the theorem.
\end{proof}

The same reasoning which proves Corollary \ref{corr:factors-are-sticky} gives the
following consequence.

\begin{corollary}\label{corr:contracts-are-sticky}
Contracts of a sticky polymatroid are sticky.\qed
\end{corollary}


\section{An information-theoretical inequality}\label{sec:inequality}

Polymatroids are used extensively when attacking problems connected to
information theory or secret sharing, see
\cite{beimel-padro,farre-padro,fmadhe,M.entreg}. Typically these problems
are concerned with certain linear inequalities which hold in every
polymatroid, or in certain subclass of polymatroids. The inequality stated
in Lemma \ref{lemma:inequality} is behind all
existing proofs that a (poly)matroid is not sticky.

In this section the usual information-theoretical abbreviations will be
used. For arbitrary subsets $I,J,K$ of the ground set we write
\begin{align*}
    f(I,J|K) &= f(IK)+f(JK)-f(K)-f(IJK), \\[2pt]
    f(I,J)   &= f(I)+f(J)-f(IJ), \\[2pt]
    f(I|K)   &= f(IK)-f(K).
\end{align*}
If $I,J,K$ are disjoint, then $f(I,J|K)$ is just the modular defect of $IK$
and $JK$. However, no disjointness is assumed in these notations. If the 
function $f$ is clear from the context, it is also omitted.
The \emph{common information} and the \emph{Ingleton expression} are defined
as follows:
\begin{align*}
  \COMM(A,B;Y) &= (A,B|Y) + (Y|A)+ (Y|B)+(Y|AB),\\[2pt]
  \ING(A,B;P,Q) &= -(A,B)+(A,B|P)+(A,B|Q)+(P,Q).
\end{align*}
It is clear from the definition that $\COMM$ is always non-negative. If it
is zero, then $Y$ is determined by both $A$ and $B$ -- this is the usual way
to express the fact that $YA$ and $A$ as well as $YB$ and $B$ have the same rank --, moreover
$A$ and $B$ are independent given $Y$. The information-theoretic
interpretation is that $Y$ contains all information that $A$ and
$B$ both have but nothing more. If $\COMM(A,B;Y)=0$ then we say that $Y$
\emph{extracts the common information} of $A$ and $B$.

The Ingleton expression $\ING$ plays an important role in matroid
representation \cite{oxley} and in
polymatroid classification \cite{ingleton,M.entreg,matus-studeny}. An
\emph{Ingleton-violating polymatroid} is one where the Ingleton expression is
negative. 

The inequality stated in Lemma \ref{lemma:inequality} appeared in
\cite{one-adhesive} and goes back to \cite{fmadhe}. Essentially it says that
if a pair can be extended to an Ingleton-violating quadruple, then one
cannot extract their common information. 

\begin{lemma}\label{lemma:inequality}
The following inequality holds for arbitrary subsets
$A,B,P,Q,Y$:
\begin{equation}\label{eq:inequality}
\ING(A,B;P,Q)+\COMM(A,B;Y)\ge 0.
\end{equation}
\end{lemma}

In Section \ref{sec:main-theorem}
we need the conditional version which uses the conditional
$\COMM$ and $\ING$ expressions defined as
\begin{align*}
   \COMM(A,B;Y\,|\,E) &= (A,B|YE)+ (Y|AE)+ (Y|BE)+(Y|ABE), \\[2pt]
   \ING(A,B;P,Q\,|\,E) &= -(A,B|E)+ (A,B|PE)+ (A,B|QE)+(P,Q|E).
\end{align*}

\begin{lemma}\label{lemma:cond-inequality}
The following inequality holds for arbitrary subsets $A,B,P,Q,Y$ and $E$:
\begin{equation}\label{eq:cond-inequality}
\ING(A,B;P,Q\,|\,E)+\COMM(A,B;Y\,|\,E)\ge 0.
\end{equation}
\end{lemma}
Before giving a proof let us remark that the conditional and unconditional
versions are equivalent. Setting $E=\emptyset$ in (\ref{eq:cond-inequality})
gives (\ref{eq:inequality}), while applying (\ref{eq:inequality}) in the
contracted polymatroid $(f,M)\contr E$ gives the conditional version
(\ref{eq:cond-inequality}). Thus it suffices to prove one of them, and we
choose the conditional version.

\begin{proof}
It is enough to show that for all subsets $A,B,P,Q,Y, E$ the expression
\begin{align*}
   & -(A,B|E)+(A,B|PE)+(A,B|QE)+(P,Q|E) + {} \\
   &~~ {}+ (A,B|YE)+(Y|AE)+(Y|BE)+(Y|ABE)
\end{align*}
is always non-negative.
The first line is the conditional $\ING$ and the second line is the
conditional $\COMM$. Expressing everything as 
linear combinations of subset ranks, this expression turns out to be the same as the 
following ten-term sum:
\begin{align*}
   & (A,B|PYE) + (A,B|QYE) + (P,Q|YE) + {} \\
   & ~~ {} + (P,Y|AE)+(P,Y|BE)+(Q,Y|AE)+(Q,Y|BE) + {} \\
   & ~~ {} + (Y | ABPE) + (Y | ABQE) + (Y | PQE).
\end{align*}
As each term in this latter sum is non-negative, inequality
(\ref{eq:cond-inequality}) holds.
\end{proof}

It is worth to note that if a (poly)matroid is representable over some
field, then the common information can always be extracted by adjoining the
intersection of the subspaces representing $A$ and $B$. Consequently in
representable (poly)matroids the Ingleton expression is always
non-negative, which was the original motivation for creating it
\cite{ingleton}.

\section{Proof of the main theorem}\label{sec:main-theorem}

This section proves the main theorem of this paper: \emph{A polymatroid with
a non-principal modular cut is not sticky}. Suppose $(f,M)$ is such a
polymatroid. Using the convolution technique we construct two extensions,
and then using a result from Section \ref{sec:inequality} we show that these
extensions have no amalgam. First we make some preparations. As explained in
Section \ref{subsec:polymatroid}, fix the non-modular pair of flats $F_1,F_2$ 
such that the modular cut $\M$ is generated by them, their
intersection $S=F_1\cap F_2$ is not in $\M$, and
$\delta = \delta(\M)=\delta(F_1,F_2)>0$.

\medskip
The first extension extracts the common information of $F_1$ and
$F_2$ using a one-element extension.
Consider the lattice $\LL$ on subsets of $N=M\cup\{a\}$ where $a\notin M$, consisting of
\begin{itemize}
\item[a)] all flats of $M$,
\item[b)] the sets $aF$ for $F\in\M$.
\end{itemize}
Lattice operations are derived from the usual subset ordering. It is easy to
check that any two elements of $\LL$ have a greatest lower bound and a least
upper bound.
Fix a non-negative $\eps\ge 0$. The rank function on $\LL$ is defined as
$$\begin{array}{l@{\,=\,}ll}
   r(F) & f(F)       & \mbox{ if $F$ is a flat in $M$},\\[2pt]
   r(aF) & f(F)+\eps & \mbox{ if $F\in\M$}.
\end{array}$$
Define the measure by $\mu(i)=f(i)$ for $i\in M$ and $\mu(a)=\eps+\delta$.
\begin{claim}\label{claim:1.1}
Conditions {\upshape(\ref{eq:convolution.1})} and
{\upshape(\ref{eq:convolution.2})} hold.
\end{claim}
\begin{proof}
This is clear when $Z_1$ and $Z_2$ are flats of $M$, so we may assume that
at least one of them contains the extra element $a$.
(\ref{eq:convolution.2}) holds as $\eps\le\mu(a)$, and
(\ref{eq:convolution.1}) follows from the fact that $Z_1\land Z_2=Z_1\cap
Z_2$ except for $aF_1\land aF_2 = F_1\cap F_2$ when $F_1,F_2\in\M$ but
$F_1\cap F_2\notin\M$. In that case the additional term is $\mu(a)$, and
(\ref{eq:convolution.1}) holds as
$$
   r(aF_1)+r(aF_2)-r(aF_1\lor aF_2)-r(F_1\cap F_2) = \eps+\delta(F_1,F_2)\ge
\mu(a),
$$
as required.
\end{proof}
\begin{claim}\label{claim:1.2}
With $N=\{a\}\cup M$ conditions {\upshape(\ref{cond:1}) -- (\ref{cond:5})} hold.
\end{claim}
\begin{proof}
An easy case by case checking.
\end{proof}

Let $(r',aM)$ be the convolution of $(r,\LL)$ and $\mu$. 
\begin{claim}\label{claim:1.3}
{\upshape a)} $(r',aM)$ is a one-point extension of $(f,M)$.
{\upshape b)} $r'(aF)=f(F)+\eps$ for all $F\in\M$.
{\upshape c)} $r'(aS)=f(S)+\mu(a)$ for the intersection $S=F_1\cap F_2$.
\end{claim}
\begin{proof}
a) is a consequence of Lemma \ref{lemma:conv-extend-poly} and Claim
\ref{claim:1.2}. b) follows from Lemma \ref{lemma:keep-lattice-rank}
and Claim \ref{claim:1.1}, as $r'(aF)=r(aF)=f(F)+\eps$.
For c) the definition of the convolution gives
$$
   r'(aS) = \min\,\{ r(Z)+\mu(aS{-}Z) \,:\, Z\in\LL \}.
$$
If $Z\in\LL$ is a flat of $M$, then $f(S)\le f(Z)+\mu_f(S{-}Z)$ by
submodularity, $\mu_f(S{-}Z)=\mu(S{-}Z)$, thus
$$
   r(Z)+\mu(aS{-}Z) \ge f(S)+\mu(a)
$$
with equality when the flat $Z$ is just $S$. If $F\in\M$, then $S\subseteq F$
and $f(F)-f(S)>\delta$ by Lemma \ref{lemma:mod-cut-estimated}, thus
$$
  r(aF)+\mu(aS{-}aF) = f(F)+\eps > f(S)+\delta+\eps=f(S)+\mu(a).
$$
In summary, $r'(aS)=f(S)+\mu(a)$, as was claimed.
\end{proof}

\begin{lemma}\label{lemma:extract-common-info}
There is a one-point extension of $(f,M)$ which extracts the conditional common
information of $F_1$ and $F_2$, that is, $\COMM(F_1,F_2,a\,|\,S)=0$.
\end{lemma}
\begin{proof}
Apply the above procedure with $\eps=0$, i.e., $\mu(a)=\delta$ to get the
polymatroid $(r',aM)$. According to Claim \ref{claim:1.3} this is
an extension of $(f,M)$, $r'(aF)=f(F)$ for
all $F\in\M$, in particular this is true for $F_1$, $F_2$, and $F_1F_2$; and 
$r'(aS)=f(S)+\delta$. Consequently
\begin{align*}
   r'(F_1,F_2|aS) &= r'(aF_1)+r'(aF_2)-r'(aF_1F_2)-r'(aS) = {} \\
    & ~~~~~~~=\delta(F_1,F_2)-\delta = 0,\\[2pt]
   r'(a|F_iS) &= r'(aF_i) - r'(F_i) = 0, \\[2pt]
   r'(a|F_1F_2S) &= r'(aF_1F_2)-r'(F_1F_2) = 0.
\end{align*}
Therefore all terms in $\COMM(F_1,F_2,a\,|\,S)$ are zero, proving the claim.
\end{proof}

\medskip
The second extension is a two-point extension of $M$ in which the
conditional Ingleton expression $\ING(F_1,F_2,u,v\,|\,S)$ is negative. 
To this end define the lattice $\LL$ on subsets of $uvM$ to consist of
\begin{itemize}
\item[a)] all flats of $M$,
\item[b)] subsets $uF$ and $vF$ for $F\in\M$,
\item[c)] $uvM$.
\end{itemize}
Choose $\eps=f(M)-\min\{\, f(F):F\in\M\}>0$. The rank function on $\LL$ is
$$\begin{array}{l@{\,=\,}ll}
   r(F) & f(F)       & \mbox{ if $F$ is a flat in $M$},\\[2pt]
   r(uF)=r(vF) & f(F)+\eps & \mbox{ if $F\in\M$},\\[2pt]
   r(uvM)      & f(M)+\eps.
\end{array}$$
Define the measure by $\mu(i)=f(i)$ for $i\in M$, and
$\mu(u)=\mu(v)=\eps+\delta$.

\begin{claim}\label{claim:2.1}
Conditions {\upshape(\ref{eq:convolution.1})} and
{\upshape(\ref{eq:convolution.2})} hold.
\end{claim}
\begin{proof}
As (\ref{eq:convolution.2}) clearly holds we turn to
(\ref{eq:convolution.1}).
Using Claim \ref{claim:1.1}, it should only be checked for
incomparable pairs $uF_1, vF_2$ with $F_1,F_2\in\M$.
In this case $uF_1\lor vF_2=uvM$ and
$uF_1\land vF_2=F_1\cap F_2$, thus we need
$$
    r(uF_1)+r(vF_2) \ge r(uF_1\lor vF_2)+r(uF_1\land vF_2)+\mu(\emptyset),
$$
which rewrites to
$$
    \delta(F_1,F_2)+f(F_1\lor F_2)+\eps \ge f(M).
$$
As $F_1\lor F_2\in\M$, this holds by the choice of $\eps$.
\end{proof}
\begin{claim}\label{claim:2.2}
With $N=\{u,v\}\cup M$ conditions {\upshape(\ref{cond:1}) -- (\ref{cond:5})}
hold.
\end{claim}
\begin{proof}
Similar to Claim \ref{claim:1.2}, a case by case checking.
\end{proof}

Let $(r',uvM)$ be the convolution of $(r,\LL)$ and $\mu$. 
\begin{claim}\label{claim:2.3}
{\upshape a)} $(r',uvM)$ is an extension of $(f,M)$.
{\upshape b)} $r'(uF)=r'(vF)=f(F)+\eps$ for all $F\in\M$.
{\upshape c)} $r'(uS)=r'(vS)=f(S)+\mu(u)$ for the intersection $S=F_1\cap
F_2$.
{\upshape d)} $r'(uvS)=\min\,\{\mu(uv)+f(S),r(uvM)\}$.
\end{claim}
\begin{proof}
a) is a consequence of Lemma \ref{lemma:conv-extend-poly} and Claim
\ref{claim:2.2}, b) follows from Lemma \ref{lemma:keep-lattice-rank}
and Claim \ref{claim:2.1}. c) is the same computation as in Claim
\ref{claim:1.3}, the only lattice element not considered there is $uvM$.
This, however, cannot yield the minimum as
$$
    r(uvM)+\mu(uS{-}uvM) = r(uM)-\mu(uS{-}uM).
$$
Finally, $r'(uvS)$ is the minimum of $r(Z)+\mu(uvS{-}Z)$ as $Z$ runs over
the element of $\LL$. If $F$ is a flat of $M$, then
$$
   r(F)+\mu(uvS{-}F) = \mu(uv)+f(F)+\mu(S{-}F) \ge \mu(uv)+f(S)
$$
with equality when $F=S$. If $F\in\M$, then $S\subseteq F$, and
\begin{align*}
   r(uF)+\mu(uvS{-}uF) & = f(F)+\eps+\mu(v)  > {} \\
                       & > f(S)+\delta+\eps+\mu(v) = \mu(uv)+f(S)
\end{align*}
according to Lemma \ref{lemma:mod-cut-estimated}. As $\mu(uvS{-}uvM)=0$, it
proves claim d).
\end{proof}

\begin{lemma}\label{lemma:negative-ingleton}
There is a two-point Ingleton-violating extension of $(f,M)$ such that
$\ING(F_1,F_2;u,v\,|\,S)<0$.
\end{lemma}
\begin{proof}
The above procedure gives the polymatroid $(r',uvM)$ which is an extension
of $(f,M)$ by Claim \ref{claim:2.3}. The terms in the conditional Ingleton 
expression can be computed as
\begin{align*}
   -r'(F_1,F_2\,|,S) &= -\delta(F_1,F_2) = -\delta, \\[2pt]
    r'(F_1,F_2\,|uS) &= r'(uF_1)+r'(uF_2)-r'(uF_1F_2)-r'(uS) = {} \\
                     & ~~~~~~~=\delta(F_1,F_2)+\eps-\mu(u) = 0, \\[2pt]
    r'(F_1,F_2\,vS) &= r'(vF_1)+r'(vF_2)-r'(vF_1F_2)-r'(vS) = 0, \\[2pt]
    r'(u,v\,|,S)    &= r'(uS)+r'(vS)-r'(uvS)-r'(S) = {} \\
                    &= f(S)+\mu(u)+f(S)+\mu(v)-r'(uvS)-f(S) = {} \\
                    &= \mu(uv)+f(S) - \min\,\{\mu(uv)+f(S),r(uvM)\}.
\end{align*}
If the minimum is taken by the first term $\mu(uv)+f(S)$, then
$r'(u,v|S)=0$, and then the Ingleton value is $-\delta<0$.
If the minimum is $r(uvM)=f(M)+\eps$, then the Ingleton value is
$$
   -\delta+\mu(uv)+f(S)-\big(f(M)+\eps\big) = \delta+\eps+f(S)-f(M).
$$
We claim that this is strictly negative. Indeed, suppose $F\in\M$ is so that
$f(M)-f(F)=\eps$, then
\begin{align*}
    f(M)-f(S) &=\big( f(M)-f(F)\big) + \big( f(F)-f(S)\big) ={} \\
              &= \eps + \big(f(F)-f(S)\big) > \eps +\delta
\end{align*}
by Lemma \ref{lemma:mod-cut-estimated}, which proves the lemma.
\end{proof}

\medskip
\begin{proof}[Proof of the main theorem]
Suppose $(f,M)$ has a non-principal modular cut. Choose flats $F_1$ and
$F_2$ such that the modular cut $\M$ generated by $F_1$ and $F_2$ has
modular defect $\delta = \delta(\M)=\delta(F_1,F_2)$, and $F_1\cap
F_2\notin\M$. By Lemma \ref{lemma:extract-common-info} $(f,M)$ has a
one-point extension $(f_1,aM)$ which extracts the conditional common
information of $F_1$ and $F_2$. By Lemma \ref{lemma:negative-ingleton} there
is two-point extension $(f_2,uvM)$ which adds an Ingleton-violating pair
$uv$. Finally, by Lemma \ref{lemma:cond-inequality}, $f_1$ and $f_2$ have no
amalgam.
\end{proof}

\noindent
\emph{Remark 1.} We proved Theorem
\ref{thm:main-result} for polymatroids, but with small modification the proof
works for matroids as well. If the given polymatroid is integer, then the
extensions constructed here are integer polymatroids as
well, thus are factors of matroids by Corollary \ref{corr:poly-to-matroid}.
If the starting polymatroid is a matroid, then these are matroid extensions,
and have no amalgam as the inequality in Lemma
\ref{lemma:cond-inequality} holds for all subsets, in particular for the
expanded high-rank atoms of the polymatroids.

\smallskip\noindent
\emph{Remark 2.} The proof above used the conditional version of inequality
(\ref{eq:inequality}). We could use the unconditional one by taking first the
contract of the polymatroid $(f,M)$ along $S$. By Corollary
\ref{corr:contracts-are-sticky} if the contract is not sticky neither is
$(f,M)$. It means that we could assume the flats $F_1$ and $F_2$ being
disjoint. We decided not to make this simplifying step to show that the
convolution method works seamlessly in the original setting.

\section{Conclusion}\label{sec:conclusion}

We have proposed a novel polymatroid construction using a ranked lattice
and a discrete measure over some finite set $M$. We proved some basic
properties of the convolution and illustrated its power by giving short 
and transparent proofs
for the interchangeability of certain matroid operations and extensions.
The following consequence of Theorem \ref{thm:factor-extension} might be of
independent interest: given
any subset $A\subset M$, one can replace the submatroid on $A$ by any other
matroid with the same rank keeping the matroid structure outside $A$
intact.

Given any (poly)matroid with a non-principal modular cut, the convolution 
technique was used to construct two extensions: one which
extracted the common information of the non-modular pair of flats, and the
second one which added an Ingleton-violating pair. These extension cannot
have an amalgam, as it would violate the (information-theoretic) inequality
proved in Section \ref{sec:inequality}.

The sticky matroid conjecture says that if a matroid has a non-modular pair
of flats then it is not sticky. We proved that if a matroid has a
non-principal modular cut then it is not sticky. The natural question
arises: is there any (poly)matroid which has a non-modular pair of flats, but
in which all modular cuts are principal? A ``no'' answer would settle the
sticky matroid conjecture.

\section*{Acknowledgments}

The author would like to thank the generous support of 
the Institute of Information Theory and Automation of the CAS, Prague.
The research reported in this paper was supported by GACR project
number 19-04579S, and partially by the Lend\"ulet program of the HAS.


\end{document}